\documentclass[12pt]{elsarticle}

\usepackage[margin=1in]{geometry}  
\usepackage{graphicx}              
\usepackage{amsmath}               
\usepackage{amsfonts}              
\usepackage{amssymb}
\usepackage{amsthm}                
\usepackage{mathtools}		
\usepackage{enumitem}		
\usepackage{hyperref}

\newtheorem{thm}{Theorem}[section]
\newtheorem{lem}[thm]{Lemma}

\newtheorem{cor}[thm]{Corollary}

\theoremstyle{definition}
\newtheorem{defi}[thm]{Definition}

\newcommand{\xb}{\mathbf{x}}

\newcommand{\cdb}{\mathcal{C}}

\begin{document}

\nocite{*}

\title{On families of bivariate copulas and their interrelation with the Hilbert space $\ell_2$ and the 
Hilbert cube $\mathcal{H}$}

\author[JFS]{Juan Fern\'andez S\'anchez}
\ead{jfsjufesa@gmail.com}
\author[WT]{Wolfgang Trutschnig\corref{cor1}}
\ead{wolfgang@trutschnig.net}

\address[JFS]{Research Group of Theory of Copulas and Applications, University of Almería, 04120 Almería, Spain}
\address[WT]{Department for Artificial Intelligence and Human Interfaces, University of Salzburg, 5020 Salzburg, Austria}

\cortext[cor1]{Corresponding author}

\begin{abstract}
The Markov kernel based metric $D_1$ was introduced in 2011 in order to construct the 
scale-invariant dependence measure $\zeta_1$, which assign each bivariate copula $C$ 
a dependence value in $[0,1]$, with $0$ exclusively for the case of independence, and $1$ 
exclusively for complete/functional dependence. In the original paper 
it has been shown that the resulting metric space $(\cdb,D_1)$ is separable and 
complete, however, no further topological properties 
were studied. Considering that $D_1$ has proved useful in a variety of contexts, using tools 
from infinite-dimensional topology, we here
close this gap, show that $(\cdb,D_1)$ is homeomorphic to the Hilbert space 
$(\ell_2,\Vert \cdot \Vert_2)$, and prove that several subfamilies are
either homeomorphic to $(\ell_2,\Vert \cdot \Vert_2)$ or to the Hilbert cube 
$(\mathcal{H},\rho)$. Moreover, allowing for a better assessment of relative sizes, we 
show that various subfamilies are so-called $Z$-sets in $(\cdb,D_1)$, implying that 
they are topologically negligible in the full space. 
\end{abstract}


\maketitle
\section{Introduction}
Constituting the link between multivariate distribution functions and their univariate margi\-nals, 
copulas have become a key tool for modeling scale-invariant dependence of 
random variables over the past decades. Apart from an amazingly broad application spectrum, copulas 
are also analytically/topologically handy objects with nice properties - they are, e.g., 
Lipschitz continuous functions and pointwise convergence coincides with uniform 
convergence, which, in turn, 
is equivalent to weak convergence of the corresponding doubly stochastic measures (see \cite{DuSe,Nel}).    
 
Driven by the need for flexible dependence models, various subfamilies of copulas have been studied
over the years, including in particular the family of Extreme-Value copulas and the family of Archimedean copulas. Working with a subclass $\mathcal{S}$ of the 
family $\cdb$ of all bivariate co\-pu\-las, the question naturally arises, how `large' 
$\mathcal{S}$ is in $\cdb$. Looking only at cardinality is certainly insufficient, the same applies 
to calculating the diameter etc. - three more promising alternatives are the following ones: \\
\textbf{(i)} Study the `size' in terms of measures of association/concordance and 
compare the range obtained 
by $\mathcal{S}$ to the one of the full family $\cdb$. Looking, for instance, at the two most famous
concordance measures Kendall's $\tau$ and Spearman's $\rho$, one could compare 
the set $\{(\tau(C),\rho(C)):\, C \in \mathcal{S} \}$ with the full region 
$\Omega:=\{(\tau(C),\rho(C)):\, C \in \cdb \}$ determined in  \cite{SPT}. 
Using this interpretation and considering Fig. 4 in \cite{SPT}, the family of Archimedean copulas
would be much `larger' than the family of Extreme Value copulas. Changing the measures of 
association, however, might result in a totally different answer, not to mention the fact 
that determining these regions usually is a non-trivial endeavor, see, e.g., \cite{KS,KM,MT}
and the references therein.\\
\textbf{(ii)} Settle for a binary classification, work with Baire categories and 
label sets as small (referred to as meager or of first category) and large 
(co-meager sets, i.e., complements of meager sets). Using this machinery one can, for instance, 
show that the family of all singular copulas is co-meager in $(\cdb,d_\infty)$ 
(see \cite{DFST.singular}), or that the family of strict Archimedean copulas is co-meager in 
the family $(\cdb_{arch},d_\infty)$ of all Archimedean copulas (see \cite{DFST.baire.symm}).  \\
\textbf{(iii)} Use tools from infinite-dimensional topology 
(see, e.g., \cite{Torun,vanMill1989}), check if a 
family $\mathcal{S} \subseteq \cdb$ is homeomorphic to the Hilbert cube 
$\mathcal{H}=[-1,1]^\mathbb{N}$ or to the Hilbert space $\ell_2$ (in which case it can be 
considered large from an absolute perspective), and compare the relative size of subsets 
in terms of so-called $Z$-sets - topologically 
negligible and hence intuitively (very) small closed sets. 
If a metric space $(\Omega,m)$ is homeomorphic
to the Hilbert cube, then it automatically has the Fixed Point property and various other 
important features (see Section \ref{sec2}); if it is homeomorphic to $\ell_2$, then it even 
is nowhere locally compact, so intuitively it might be considered very large.  
As demonstrated recently in \cite{LiuYang2023},
using this approach it can be shown, e.g., that $(\cdb,d_\infty)$ as well as 
the family of exchangeable/symmetric copulas $\cdb^{ex}$ are homeomorphic to the Hilbert 
cube $\mathcal{H}$, but the latter is a $Z$-set in $(\cdb,d_\infty)$.  

We here follow the third approach but, contrary to \cite{DFST2026}, focus 
on the Markov kernel based metric $D_1$, mentioned in \cite{LiuYang2023}, introduced 
in \cite{p06} and extended to the multivariate setting in \cite{D1multi,GJT}. 
In a nutshell, $D_1$ was introduced in \cite{p06} in order to construct a dependence measure 
$\zeta_1$ assigning each copula a dependence value in $[0,1]$, with $\zeta_1(C)=0$ 
exclusively for the independence/product copula $\Pi$, and $\zeta_1(C)=1$ if and only if 
$C$ is completely dependent (in the sense that for $(X,Y)$ having copula $C$, the random 
variable $Y$ is a function of $X$, see Section \ref{sec2}). 
According to \cite{p06}, the resulting metric space $(\cdb,D_1)$ is complete and separable, contrary 
to $(\cdb,d_\infty)$ it is, however, not compact. Much more is true: as a by-product of our first 
two main theorems we conclude that both, $(\cdb,D_1)$ as well as the family $\cdb^{mcd}$ 
of all mutually completely dependent copulas (i.e., completely dependent copulas $C$ whose 
transpose $C^t$ is completely dependent too), are nowhere locally compact.\\
 
Our main results can be summarized as follows: $(\cdb,D_1)$ is homeomorphic to 
the separable Hilbert space $(\ell_2,\Vert \cdot \Vert_2)$ of all square summable sequences 
and the same is true for $(\cdb^{mcd},D_1)$. Moreover, the family $(\cdb^{ev},D_1)$ of all 
Extreme Value copulas as well as the 
family $(\cdb^{sh},D_1)$ of all spatially homogeneous copulas are homeomorphic to the Hilbert cube.
All three families $\cdb^{mcd},\cdb^{ev},\cdb^{sh}$ are at the same time $Z$-sets (and hence very small) in the metric space $(\cdb,D_1)$.\\

The remainder of this note is organized as follows: Section \ref{sec2} collects notations 
and preliminaries on copulas and on infinite-dimensional topology. 
Section \ref{sec:results} contains the result that $(\cdb,D_1)$ is homeomorphic 
to $\ell_2$, Section \ref{sec:4} the afore-mention properties of families related with 
univariate functions. Some examples illustrate the constructions used in the proofs.        
 
\section{Notation and preliminaries}\label{sec2}
For every metric space $(\Omega,m)$ and every $r \geq 0$ we will let $B(x,r)$ and 
$\overline{B}(x,r)$ denote the open and the closed ball of radius $r$ around $x$, respectively.  
The Borel $\sigma$-field on $\Omega$ will be denoted by $\mathcal{B}(\Omega)$, the Dirac measure
at $a \in \Omega$ by $\delta_a$, the indicator function of a set $A \subseteq \Omega$ by 
$\mathbf{1}_A$. For measure spaces $(\Omega_1,\mathcal{A}_1,\mu_1), (\Omega_2,\mathcal{A}_2,\mu_2)$
and a measurable mapping $f: \Omega_1 \to \Omega_2$ we will let 
$\mu_1^f$ denote the push-forward of $\mu_1$ via $f$, i.e., $\mu_1^f(E)=\mu_1(f^{-1}(E))$ for 
every $E \in \mathcal{A}_2$. 

In what follows, $\mathcal{H}=[-1,1]^\mathbb{N}$ denotes the Hilbert cube, 
we will denote sequences and vectors by bold symbols and write, e.g., 
$\mathbf{x}=(x_1,x_2,\ldots) \in \mathcal{H}$. Defining $\rho: \mathcal{H} \times \mathcal{H} \to [0,1]$  
by $\rho(\mathbf{x},\mathbf{y})=\sum_{i=1}^\infty \frac{1}{2^i}\vert x_i-y_i \vert $ 
yields a metrization of coordinate-wise convergence, the resulting metric space   
 $(\mathcal{H},\rho)$ is convex and compact.  
The Hilbert cube $(\mathcal{H},\rho)$ is commonly viewed as `prototype' compact metric space, 
since every compact metric space $(\Omega,m)$ is homeomorphic to a compact subset 
of $(\mathcal{H},\rho)$; more generally, a metric space $(\Omega,m)$ is separable if and only if 
it can be embedded in $(\mathcal{H},\rho)$, see, e.g., \cite{vanMill1989}. \\
In the sequel $(\ell_2,\Vert \cdot \Vert_2)$ denotes the separable Hilbert space of all 
square-summable sequences $\mathbf{x} \in \mathbb{R}^\mathbb{N}$.
According to \cite{Anderson}, $(\ell_2,\Vert \cdot \Vert_2)$ is homeomorphic to 
$\mathbb{R}^\mathbb{N}$ endowed with the metric $\sigma$, given by
$$
\sigma(\mathbf{x},\mathbf{y}) := \sum_{i=1}^\infty 2^{-i} \frac{\vert x_n-y_n \vert}{1+\vert x_n-y_n },
\quad \mathbf{x},\mathbf{y} \in \mathbb{R}^\mathbb{N}.
$$
It is well-known and straightforward to verify that $\sigma$ is a metrization of 
the product topology, convergence w.r.t. $\sigma$ is equivalent to 
coordinate-wise convergence.

For the sake of completeness we here recall some main notions
and some results from infinite-dimensional topology used in the sequel, for more background we 
refer to the textbooks \cite{vanMill1989} and \cite{Sakai2020}.   

\begin{defi}[Retract]
Let $(\Omega,m)$ be a metric space and $A$ a closed subset of $\Omega$. Then $A$ is called a
\emph{retract} if there exists a continuous mapping $r: \Omega \rightarrow A$ that has 
each $x \in A$ as fixed point. Such a mapping $r$ is called a \emph{retraction}. \\
Moreever, $(\Omega,m)$ is an \emph{absolute retract} (AR, for short), if it is a
retract of every metric space containing $\Omega$ as a closed subset.
\end{defi}
\noindent It is well known (and not hard to verify) 
that every retract of an AR is an AR (see, e.g., \cite{Hu1965}). Moreover, according to \cite{Du}
every convex subset $E$ of a normed space $(N,\Vert \cdot \Vert)$ is an AR.

\begin{defi}[SDLFA property]
A separable metric space $(\Omega,m)$ has the strong discrete locally finite approximation property (SDLFA property, for short) if for every function $\varepsilon \colon \Omega \to (0,\infty)$ there exists a 
map $f \colon \Omega \times \mathbb{N} \longrightarrow \Omega$ 
satisfying  
\begin{equation}\label{ineq:approx}
m\big(f(x,n), x\big) < \varepsilon(x) 
\end{equation}
for every $(x,n) \in \Omega \times \mathbb{N}$ and fulfilling that the family
\[
\big\{\, f(\Omega \times \{n\}) : n \in \mathbb{N} \,\big\}
\]
is locally finite in $\Omega$, i.e., every $x \in \Omega$ belongs to only finitely many of the 
sets $f(\Omega \times \{n\})$.
\end{defi}

\begin{thm}[Toru\'nczyk characterization for $\ell_2$]\label{thm:torun.l2}
A complete and separable metric space $(\Omega,m)$ is homeomorphic to the Hilbert space 
$(\ell_2,\Vert \cdot \Vert_2)$ if and only if 
it is an AR and has the strong discrete locally finite approximation property.
\end{thm} 

\begin{defi}[$Z$-set]
A subset $A$ of $(\Omega,m)$ is a called a $Z$-set if it is closed and for every
continuous mapping $\varepsilon: \Omega \rightarrow (0,1)$ there exists some continuous function 
$f: \Omega \rightarrow \Omega\setminus A$ such that $m(x,f(x)) < \varepsilon(x)$ holds for all
$x \in \Omega$. 
\end{defi}
In the case $(\Omega,m)$ is compact, it obviously suffices to consider $\varepsilon$ constant.
We interpret $Z$-sets as topologically negligible (and hence topologically very small) sets. 
Obviously $Z$-sets are nowhere dense (hence of first Baire category), but not vice versa, so 
being a $Z$-set is a strictly stronger condition than being of first category. 

\begin{defi}[Fixed Point Property]
We say that a metric space $(\Omega,m)$ has the Fixed Point property (FP property, for short) if every continuous mapping $f: \Omega \to \Omega$
has at least one fixed point.
\end{defi}

\begin{defi}[Disjoint Cells Property]
We say that a metric space $(\Omega,m)$ has the \emph{disjoint-cells property} (DC property, for short) 
if for every $n \in \mathbb{N}$, every continuous
function $f: [0,1]^n \times \{0,1\} \to \Omega$, and every $\varepsilon > 0$, there exists a 
continuous function $g: [0,1]^n \times \{0,1\} \to \Omega$, such that
$$
\sup\{m(f(\xb),g(\xb)):\, \xb \in [0,1]^n \times \{0,1\} \} < \varepsilon \quad \text{and} \quad 
g([0,1]^n \times \{0\}) \cap g([0,1]^n \times \{1\}) = \emptyset.
$$
\end{defi}

\begin{thm}[Toru\'nczyk characterization of $\mathcal{H}$]\label{thm:torun}
A metric space $(\Omega,m)$ is homeomorphic to $\mathcal{H}$ if and only if it is compact,
an $AR$, and has the DC property.
\end{thm} 
For convex subsets of normed spaces, the following equivalence (in which $\operatorname{ind}$ 
denotes the small inductive dimension) holds:
\begin{thm}[Keller--Dobrowolski--Toru\'nczyk]\label{thm:keller}
Let $E$ be a convex subset of a normed space $(N,\Vert \cdot \Vert)$. Then $E$ is homeomorphic to
$\mathcal{H}$ if and only if $E$ is compact and $\operatorname{ind}E=\infty$. 
\end{thm}

As a direct consequence of Theorem \ref{thm:torun} and Theorem \ref{thm:keller}, 
$(\mathcal{H},\rho)$ is an AR and has the DC property, and the same 
holds true for any metric space $(\Omega,m)$ homeomorphic to $(\mathcal{H},\rho)$.
Moreover, $(\mathcal{H},\rho)$ has the FP property, the latter is homeomorphism invariant, so 
every me\-tric space $(\Omega,m)$ homeomorphic to $(\mathcal{H},\rho)$ has the FP property as well. 
In contrast, $(\ell_2,\Vert \cdot \Vert_2)$ does not have the FP property, a standard counterexample 
is the mapping $\mathbf{x} \mapsto \mathbf{x} + \mathbf{y}$ for some fixed 
$\mathbf{y} \in \ell_2\setminus \{\mathbf{0}\} $. \\

Before deriving the results mentioned in the introduction we recall some basic facts about 
bivariate copulas and doubly stochastic measures. 
$\cdb$ will denote the family of all bivariate copulas. For each copula $C$ the corresponding 
doubly stochastic measure will be denoted by $\mu_C$, i.e.,\ $\mu_C([0,x]\times[0,y]) = C(x,y)$ 
for all $x,y \in [0,1]$; $\mathcal{P}_{\cdb}$ will denote the family of all doubly stochastic measures.
$M$ will denote the minimum copula, $W$ the Fr\'echet--Hoeffding bound and $\Pi$ the product copula. 
 For more background on copulas and doubly stochastic measures we refer to \cite{DuSe,Nel}. 

Since the main subject of this paper, the metric $D_1$, is based on conditional distributions, 
Markov kernels play an important role.
A Markov kernel from $\mathbb{R}$ to $\mathbb{R}$ is a mapping $K : \mathbb{R} \times \mathcal{B}(\mathbb{R}) \to [0,1]$ such that for every fixed $E \in \mathcal{B}(\mathbb{R})$ the mapping $x \mapsto K(x,E)$ is Borel-measurable and for every fixed $x \in \mathbb{R}$ the mapping $E \mapsto K(x,E)$ is a probability measure. Given two arbitrary, real-valued random variables $X, Y$ on a joint 
probability space $(\Omega, \mathcal{A}, \mathbb{P})$, we refer to a Markov kernel $K$ as 
regular conditional distribution of $Y$ given $X$, if
$$
K(X(\omega), E) = E(\mathbf{1}_E \circ Y \mid X)(\omega)
$$
holds $\mathbb{P}$-almost surely for every $E \in \mathcal{B}(\mathbb{R})$. 
It is well-known (see, e.g., \cite{Kallenberg}) that for $X, Y$ as above, a regular 
conditional distribution of $Y$ given $X$ always exists and is unique for $\mathbb{P}^X$-a.e.
 $x \in \mathbb{R}$. If $(X,Y)$ has distribution function $F$ (in which case we will also 
 write $(X,Y) \sim F$ and let $\mu_F$ denote the corresponding probability measure on $\mathcal{B}(\mathbb{R}^2)$) we will let $K_F$ denote (a version of) the regular conditional distribution of $Y$ given $X$ and simply refer to it as Markov kernel of $F$. If $C$ is a copula then we will consider the Markov kernel of $C$ automatically as mapping $K_C : [0,1] \times \mathcal{B}([0,1]) \to [0,1]$. Defining the $x$-section of a set $G \in \mathcal{B}(\mathbb{R}^2)$ as $G_x := \{y \in \mathbb{R} : (x,y) \in G\}$ the so-called disintegration theorem (see \cite{Kallenberg}) yields
\begin{equation}
\int_{\mathbb{R}} K_F(x, G_x)\, d\mathbb{P}^X(x) = \mu_F(G).
\end{equation}
As a direct consequence, for every $C \in \mathcal{C}$ we get
$$\int_{[0,1]} K_C(x,E)\, d\mathbb{P}^X(x) = \int_{[0,1]} K_C(x,E)\, d\lambda(x) = \lambda(E)$$
for every $E \in \mathcal{B}([0,1])$, whereby $\lambda$ denotes the Lebesgue measure on $\mathbb{R}$.
In what follows we will also work with the restriction of $\lambda$ to a Borel set $E$, which 
we will denote by $\lambda_E$. For more background on conditional expectation and general disintegration we refer to \cite{Kallenberg}.

We call a copula $C$ completely dependent if there exists a $\lambda$-preserving transformation 
$h : [0,1] \to [0,1]$ (i.e., a transformation fulfilling $\lambda^h=\lambda$)
 such that $K(x,E) := \mathbf{1}_E(h(x))$ is a Markov kernel of $C$. 
The family of all $\lambda$-preserving transformations $h: [0,1] \to [0,1]$ will be
denoted by $\mathcal{T}$, the subclass of all $\lambda$-preserving bijections by $\mathcal{T}^*$. 
We will write $\cdb^{cd}$ and $\cdb^{mcd}$ for the corresponding families of 
completely dependent and mutually completely dependent copulas, respectively.  
In the sequel we will view $\mathcal{T}$ and $\mathcal{T}^*$ as subsets of the Banach space 
$L_1([0,1]):=L_1([0,1],\mathcal{B}([0,1]),\lambda)$, i.e., 
equivalence classes will be considered.  
For more properties of (mutual) complete dependence we refer to \cite{p06} and the references therein.

Markov kernels were used in \cite{p06} in order to construct metrics stronger than the standard uniform metric $d_\infty$, defined by
\begin{align}
d_\infty(C_1, C_2) := \max_{(x,y) \in [0,1]^2} |C_1(x,y) - C_2(x,y)|
\end{align}
on $\mathcal{C}$. It is well known that the metric space $(\mathcal{C}, d_\infty)$ is compact and that pointwise and uniform convergence of a sequence of copulas $(C_n)_{n \in \mathbb{N}}$ are equivalent (see \cite{DuSe, Nel}). Following \cite{p06} and defining
\begin{align}
D_1(C_1, C_2) &:= \int_{[0,1]} \int_{[0,1]} \big|K_{C_1}(x,[0,y]) - K_{C_2}(x,[0,y])\big|\, d\lambda(x)\, d\lambda(y), \quad C_1,C_2 \in \cdb,   
\end{align}
it can be shown that $D_1$ is a metric and that the resulting metric space 
$(\cdb,D_1)$ is sepa\-rable and complete, but not compact, and has a diameter of $\frac{1}{2}$. 
Moreover, we have that $D_1(C, \Pi) \in \big[0, \tfrac{1}{3}\big]$ for every $C \in \mathcal{C}$ 
and that  $D_1(C,\Pi)$ is maximal if and only if $C$ is completely dependent. The metric $D_1$ was originally introduced in order to construct a dependence measure which, contrary to $d_\infty$, is capable of separating independence and complete dependence. The resulting $D_1$-based dependence measure $\zeta_1$ introduced in \cite{p06} is defined as
\begin{equation}
\zeta_1(C) := 3 \cdot D_1(C, \Pi), \quad C \in \cdb. 
\end{equation}
In the continuous setting, Chatterjee's by now famous dependence measure 
(see \cite{Chatterjee}) boils down to the $L_2$-version of $\zeta_1$, i.e., 
a normalized version of $D_2(C,\Pi)$ is considered.  
As shown in \cite{p06}, replacing $D_1$ by $D_p$, defined analogously by
\begin{align}
D_p^p(C_1, C_2) &:= \int_{[0,1]} \int_{[0,1]} \big|K_{C_1}(x,[0,y]) - K_{C_2}(x,[0,y])\big|^p\, d\lambda(x)\, d\lambda(y), \quad p \in [1,\infty),  
\end{align}
yields a family of metrics on $\cdb$, which are not equivalent (as metrics) but all
induce the same topology on $\cdb$. As a direct consequence, all results derived in this article 
for $D_1$ also hold for every $D_p$ with $p \in [1,\infty)$.

\section{Results for the full metric space}\label{sec:results}
We first study topological properties of the whole metric space $(\cdb,D_1)$,  
want to show that $(\cdb,D_1)$ is homeomorphic to $(\ell_2,\Vert \cdot \Vert_2)$, and 
start with the following lemma.
\begin{lem}\label{lem:D1AR}
$(\cdb,D_1)$ is an AR.
\end{lem}
\begin{proof}
We show that $(\cdb,D_1)$ is homeomorphic to a convex (and closed) subset $E$ of 
the Banach space $L_1([0,1]^2)$. 
Define the mapping $\kappa: \cdb \to L_1([0,1]^2)$ by 
\begin{equation}\label{eq:embed.to.L1}
(\kappa(A))(x,y):=K_A(x,[0,y])=:H_A(x,y)
\end{equation}
and set $E:=\kappa(\cdb)$. Then obviously $\kappa$ is an isometry between 
$(\cdb,D_1)$ and $(E,\Vert \cdot \Vert_1)$, which preserves convex combinations, 
so $\kappa$ is injective and continuous
and the same holds for its inverse $\kappa^{-1}: (E,\Vert \cdot \Vert_1) \to (\cdb,D_1)$.
As a direct consequence, $\kappa$ is a homeomorphism between 
$(E,\Vert \cdot \Vert_1)$ and $(\cdb,D_1)$. Since $E$ is as convex subset of 
$L_1([0,1]^2)$ an AR, it follows that so is $(\cdb,D_1)$ and the proof is complete.
\end{proof}
The following theorem provides an affirmative answer (with simple proof) 
to the first part of Problem 3 in \cite{LiuYang2023}. 
\begin{thm}\label{thm:main.D1.full}
$(\cdb,D_1)$ is homeomorphic to $\ell_2$. 
\end{thm}
\begin{proof}
It was shown in \cite{p06} that the metric space $(\cdb,D_1)$ is complete and separable. Since 
Lemma \ref{lem:D1AR} states that $(\cdb,D_1)$ is an AR, according to Theorem \ref{thm:torun.l2}
it suffices to prove that $(\cdb,D_1)$ has the SDLFA property.
For every $a \in (0,1)$ define the affine contractions $L,U: [0,1]^2 \to [0,1]^2$ by 
$$
L(x,y)=(x,ay),\quad U(x,y)=(x,a+(1-a)y).
$$
For every pair of copulas $A,B$ and every $a \in [0,1]$ let $\mathcal{G}_a(A,B)$ denote the 
unique copula $C$ whose corresponding doubly stochastic measure $\mu_C$ is given by
$$
\mu_C=a\mu_A^L + (1-a) \mu_B^U.
$$
Then obviously the Markov kernel $K_{\mathcal{G}_a(A,B)}$ of $\mathcal{G}_a(A,B)$ is given by 
\begin{equation}\label{eq:glue}
K_{\mathcal{G}_a(A,B)}(x,[0,y]) =
\begin{cases}
  a\,K_A\left(x,\left[0,\tfrac{y}{a}\right]\right)  & \text{if } x,y \in [0,1]\times [0,a], \\
  a+(1-a) K_B\left(x,\left[0,\tfrac{y-a}{1-a}\right]\right)& \text{if } x,y \in [0,1]\times (a,1].
\end{cases}
\end{equation}
Figure \ref{fig:glue} depicts a sample of size $n=10.000$ of the copula 
$\mathcal{G}_\frac{3}{4}(\Pi,M)$. 
Suppose now that $\varepsilon: \Omega \to (0,\infty)$ is an arbitrary function. 
Define $\varepsilon': (\cdb,D_1) \to (0,1]$ by $\varepsilon'(C)=\min(\varepsilon(C),1)$.  For every 
$n \in \mathbb{N}$ setting $C_n:=\frac{1}{n+1} W + \frac{n}{n+1} \Pi \in \cdb$ and 
defining the function $f: \cdb \times \mathbb{N} \to \cdb$ by
\begin{align}
f(C,n)=\mathcal{G}_{1-\frac{\varepsilon'(C)}{4}}(C,C_n),
\end{align}  
it follows immediately that for every $n_1 \neq n_2$ we have that 
$$
\{f(C,n_1): \, C \in \cdb\} \cap \{f(C,n_2): \, C \in \cdb\} = \emptyset.
$$
Hence, writing $f(\cdb,n):=\{f(C,n): \, C \in \cdb\}$ for every $n \in \mathbb{N}$, 
the family $\{f(\cdb,n): \, n \in \mathbb{N}\}$ is locally finite (in fact, even pairwise disjoint) 
in $\cdb$.  \\
It remains to show that ineq. \eqref{ineq:approx} holds. Fix $C \in \cdb$ and to simplify notation 
write $a:=1-\frac{\varepsilon'(C)}{4} \in [\frac{3}{4},1)$. 
Using eq. \eqref{eq:glue} we obviously have 
\begin{align*}
D_1\left(\mathcal{G}_{1-\frac{\varepsilon'(C)}{4}}(C,C_n),C\right) = I + II
\end{align*}
with $I$ and $II$ given by
\begin{align*}
I&= \int_{[0,1]} \int_{[0,a]} \vert a\, K_C\left(x,\left[0,\tfrac{y}{a}\right]\right) - 
K_C(x,[0,y])\vert \,d\lambda(y)d\lambda(x) \\
II&= \int_{[0,1]} \int_{[a,1]} \vert a + (1-a)K_{C_n}\left(x,\left[0,\tfrac{y-a}{1-a}\right]\right) - 
K_C([0,y]) \vert  \,d\lambda(y)d\lambda(x).
\end{align*}
For the first integral, using the triangle inequality and change of coordinates we get
\begin{align*}
I&= \int_{[0,1]} \int_{[0,a]} \left| a\, K_C\left(x,\left[0,\tfrac{y}{a}\right]\right) - 
(a+1-a)K_C(x,[0,y]) \right| \,d\lambda(y)d\lambda(x) \\
&= \int_{[0,1]} \int_{[0,a]} \left| a\, K_C\left(x,\left[0,\tfrac{y}{a}\right]\right) - 
a\,K_C(x,[0,y]) \, + (1-a)\,K_C(x,[0,y])\right| \,d\lambda(y)d\lambda(x) \\
&\leq \int_{[0,1]} \int_{[0,a]}a\, K_C\left(x,\left[0,\tfrac{y}{a}\right]\right) - 
a\,K_C(x,[0,y]) \,d\lambda(y)d\lambda(x) + a(1-a) \\
&\leq \int_{[0,1]} \int_{[0,1]}a^2\, K_C\left(x,\left[0,v\right]\right)\,d\lambda(y)d\lambda(x)  - 
\int_{[0,1]} \int_{[0,a]}a\,K_C(x,[0,y]) \,d\lambda(y)d\lambda(x) + a(1-a) \\
&\leq \int_{[0,1]} \int_{[0,1]}a\, K_C\left(x,\left[0,v\right]\right)\,d\lambda(y)d\lambda(x)  - 
\int_{[0,1]} \int_{[0,a]}a\,K_C(x,[0,y]) \,d\lambda(y)d\lambda(x) + a(1-a) \\
&= \int_{[0,1]} \int_{[a,1]} a\,K_C(x,[0,y])\,d\lambda(y)d\lambda(x) + a(1-a) \\
&\leq 2a(1-a) \leq 2(1-a).
\end{align*}
Since the integrand of $II$ is bounded from above by $1$, we have $II \leq 1-a$, which altogether yields 
\begin{align*}
D_1\left(f(C,n),C\right) = I + II \leq 3(1-a) = 3\,\tfrac{\varepsilon'(C)}{4} < \varepsilon'(C) 
\leq \varepsilon(C)
\end{align*} 
Since $C$ and $n$ were arbitrary this completes the proof.
\end{proof}
\begin{figure}[htbp]
  \centering
  \includegraphics[width=0.7\textwidth]{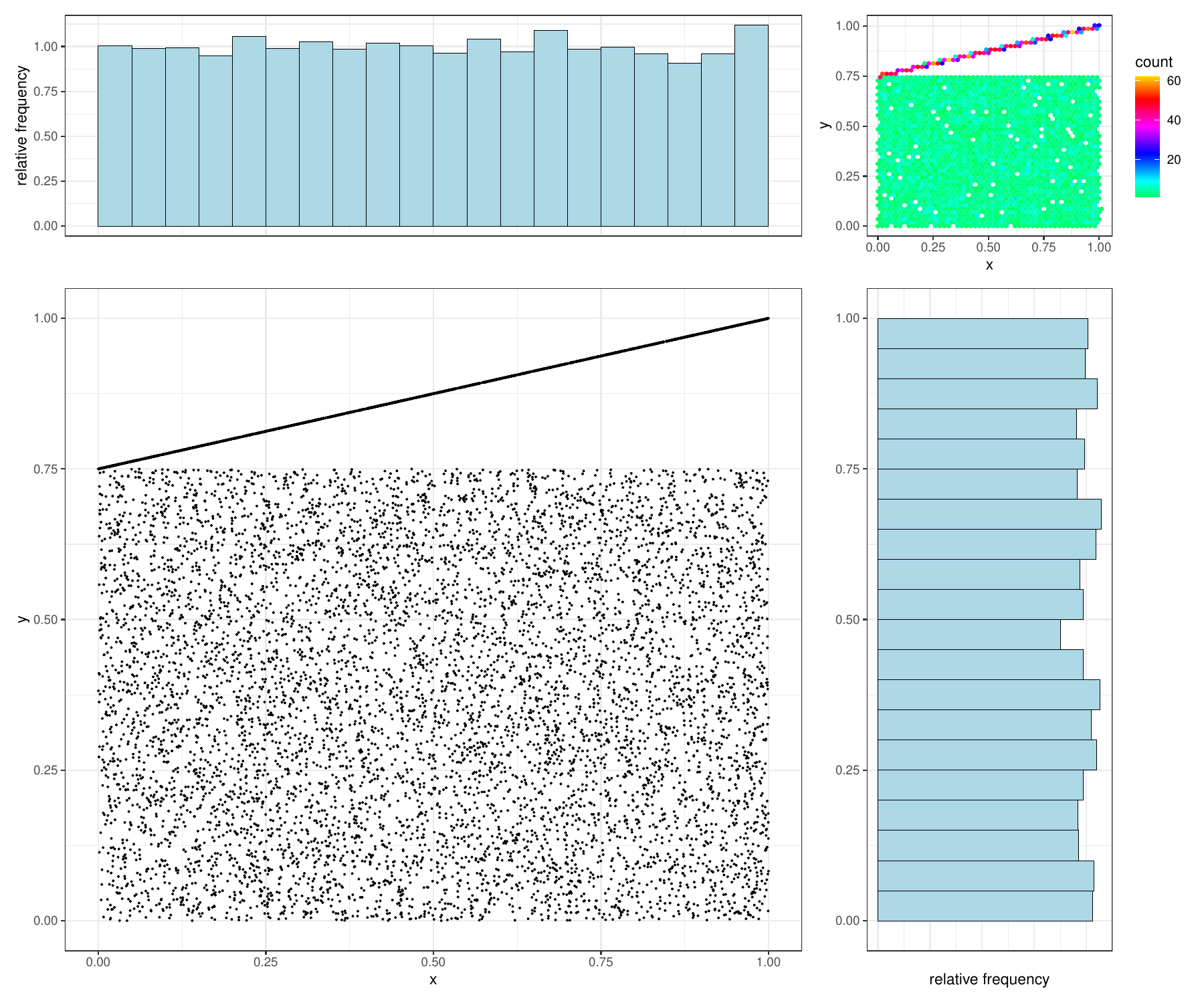}
  \caption{Sample of size $n=10.000$ of the copula $\mathcal{G}_{\frac{3}{4}}(\Pi,M)$ (upper
  left panel); corresponding marginal histograms (upper left and lower right panel)
  and heatmap of the sample (upper right panel).}
  \label{fig:glue}
\end{figure}

Using standard properties of the infinite-dimensional separable Hilbert space 
$(\ell_2,\Vert \cdot \Vert_2)$ we obtain 
the following immediate consequence, which is striking in so far, as it once more underlines 
the fundamental difference to the uniform metric $d_\infty$ for which $(\cdb,d_\infty)$ is compact:
\begin{cor}\label{cor:not.locally.compact}
For every $C \in \cdb$ and every $r>0$ the closed ball 
$$
\overline{B}(C,r):=\{A \in \cdb: D_1(A,C)\leq r\}
$$ 
is not compact in $(\cdb,D_1)$. In other words, $(\cdb,D_1)$ is nowhere locally compact.   
\end{cor} 

The fact that $(\cdb,D_1)$ is nowhere locally compact raises the question, whether 
standard nonparametric subfamilies of $(\cdb,D_1)$ can be compact. We will answer the question affirmatively for the families of Extreme Value and spatially homogeneous copulas, but first
focus on mutually completely dependent copulas. 

\section{Results for subfamilies related with univariate functions}\label{sec:4}
\subsection{The family $\cdb^{mcd}$ of all completely dependent copulas}
Although the families $\cdb^{mcd}$ and $\cdb^{cd}$ may seem quite pathological since they 
only contain copulas describing the situation 
of full predictability of a random variable $Y$ given another variable $X$, both families are
rather large topologically: 
In 1968, working with Markov operators, Kim showed that, in the language of copulas, 
$\cdb^{mcd}$ is co-meager in $(\cdb,d_\infty)$. 
Viewed as a subset of $(\cdb,D_1)$, however, $\cdb^{cd}$ (hence $\cdb^{mcd}$) 
is very small since it is nowhere dense.    
Despite being small in $(\cdb,D_1)$, the family of all mutually completely dependent on its own 
is topologically very large - in this section we will show that $(\cdb^{mcd},D_1)$ is homeomorphic to 
$(\ell_2,\Vert \cdot \Vert_2)$, implying in particular that $(\cdb^{mcd},D_1)$ is nowhere  
locally compact. Our results build upon \cite{Nhu}, where the author showed that 
$\mathcal{T}^*$ endowed with the weak topology is homeomorphic to $(\ell_2,\Vert \cdot \Vert_2)$. 
As shown in \cite{Nhu}, for $h,h_1,h_2,\ldots \in \mathcal{T}^*$ we have that 
$(h_n)_{n \in \mathbb{N}}$ converges weakly to $h$ if and only if ($\Delta$ denoting the symmetric 
difference)
\begin{equation}\label{eq:weak}
\lim_{n \rightarrow \infty} \lambda \left(h_n(B) \Delta h(B) \right)=0
\end{equation} 
holds for every $B \in \mathcal{B}([0,1])$. 
To obtain our main result it suffices to prove that the convergence in \eqref{eq:weak} is 
equivalent to $L_1$-convergence in $\mathcal{T}^*$. We start with the following lemma
considering convergence for the larger family $\mathcal{T}$:
\begin{lem}\label{lem:L1.in.prob}
For all $h,h_1,h_2,\ldots \in \mathcal{T}$ the following three assertions are equivalent:
\begin{enumerate}
\item $\lim_{n \to \infty} \Vert h_n-h \Vert_1=0$.
\item $(h_n)_{n \in \mathbb{N}}$ converges in probability to $h$, i.e., for every $\varepsilon>0$
we have 
$$
\lim_{n \to \infty}  \lambda \left(\{x \in [0,1]: |h_n(x)-h(x)| > \varepsilon\}\right)=0.
$$
\item For every $B \in \mathcal{B}([0,1])$ we have
$$
\lim_{n \rightarrow \infty} \lambda \left(h_n^{-1}(B) \Delta h^{-1}(B) \right)=0.
$$
\end{enumerate}
\end{lem} 
\begin{proof}
The equivalence of the first two assertions is straightforward to verify: on the one hand, 
using Markov/Chebyshev inequality for every $\varepsilon>0$ we have
$$
\lambda\left(\{x \in [0,1]: |h_n(x)-h(x)| > \varepsilon\}\right) \leq \tfrac{\Vert h_n-h \Vert_1}{\varepsilon}.
$$
And on the other hand, considering $|h_n(x)-h(x)| \in [0,1]$ yields
\begin{align*}
\Vert h_n-h \Vert_1 &= \int_{\{x \in [0,1]: |h_n(x)-h(x)|>\frac{\varepsilon}{2}\}} |h_n(x)-h(x)| d\lambda(x) + \tfrac{\varepsilon}{2} \\
&\leq \lambda\left(\{x \in [0,1]: |h_n(x)-h(x)|>\tfrac{\varepsilon}{2}\}\right)\, + \,\tfrac{\varepsilon}{2}. 
\end{align*} 
To show that the second assertion implies the third one we proceed as follows: For every interval
$[a,b] \subseteq [0,1]$ the function $\mathbf{1}_{[a,b]}$ obviously is continuous 
$\lambda^h=\lambda$-almost everywhere, so the continuous mapping theorem directly yields
that $(\mathbf{1}_{[a,b]} \circ h_n)_{n \in \mathbb{N}}$ converges in probability to 
$\mathbf{1}_{[a,b]} \circ h$, i.e., for every $\varepsilon>0$ 
$$
\lim_{n \rightarrow \infty}\lambda \left(h_n^{-1}([a,b]) \Delta h^{-1}([a,b]) \right) = 
\lim_{n \rightarrow \infty}\lambda\left(\{x \in [0,1]: |\mathbf{1}_{[a,b]} \circ h_n(x)-\mathbf{1}_{[a,b]} \circ h(x)|>\tfrac{\varepsilon}{2}\}\right)=0
$$ 
holds. Extending from $[a,b]$ to finite unions of intervals is straightforward. Since for each
Borel set $B \in \mathcal{B}([0,1])$ and each $\varepsilon>0$ there exists a 
finite union $U$ of intervals with $\lambda(B \Delta U) < \varepsilon$, using the facts that 
$h_n,h$ are $\lambda$-preserving and that $h_n^{-1},h^{-1}$ preserve set operations like the 
symmetric difference, the third assertions follows.   \\
Suppose now that third the assertion holds and let $\varepsilon>0$ be arbitrary but fixed. For every
index $N \in \mathbb{N}$ fulfilling $\frac{1}{N}<\varepsilon$, setting 
$I^N_i:=[\frac{i-1}{N},\frac{i}{N})$ for every $i \in \{1,\ldots,N-1\}$  as well as 
$I_n:=[\frac{N-1}{N},1]$ we have 
$$
\lim_{n \rightarrow \infty} \sum_{i=1}^N \lambda \left(h_n^{-1}(I^N_i) \, \Delta \, h^{-1}(I^N_i) \right)=0.
$$
Defining $\Lambda_n^N:=\bigcup_{i=1}^N \left(h_n^{-1}(I^N_i) \, \Delta \, h^{-1}(I^N_i)\right)$
this yields $\lim_{n \rightarrow \infty}\lambda(\Lambda_n^N)=0$. Since 
for every $x \in [0,1]$ fulfilling $|h_n(x)-h(x)|>\varepsilon$ we have that $x \in \Lambda_n^N$, it 
altogether follows that 
$$
\lim_{n \to \infty}  \lambda \left(\{x \in [0,1]: |h_n(x)-h(x)| > \varepsilon\}\right) 
\leq \lim_{n \to \infty} \lambda(\Lambda_n^N)=0.
$$
Considering that $\varepsilon>0$ was arbitrary, we obtain that 
$(h_n)_{n \in \mathbb{N}}$ converges in probability to $h$, and the proof is complete.
\end{proof}
Even though transformations $g,h \in \mathcal{T}^*$ fulfilling    
$\Vert f-g \Vert_1 \neq \Vert f^{-1}-g^{-1} \Vert_1$ exist (and are easy to construct), 
$L_1$-convergence in $\mathcal{T}^*$ is equivalent to $L_1$ convergence of the corresponding inverses
- the following lemma holds. 
\begin{lem}
For all $h,h_1,h_2,\ldots \in \mathcal{T}$ the following three assertions are equivalent:
\begin{enumerate}
\item $\lim_{n \to \infty} \Vert h_n-h \Vert_1=0$.
\item $\lim_{n \to \infty} \Vert h_n^{-1}-h^{-1} \Vert_1=0$.
\item For every $B \in \mathcal{B}([0,1])$ we have  
$
\lim_{n \rightarrow \infty} \lambda \left(h_n^{-1}(B) \Delta h^{-1}(B) \right)=0.
$
\item For every $B \in \mathcal{B}([0,1])$ we have  
$
\lim_{n \rightarrow \infty} \lambda \left(h_n(B) \Delta h(B) \right)=0.
$
\end{enumerate}
\end{lem} 
\begin{proof}
Let $B \in \mathcal{B}([0,1])$ be arbitrary but fixed and set $A=h(B)\in \mathcal{B}([0,1])$. 
Then using the fact that $h_n,h$ and their inverses are $\lambda$-preserving, it follows that
\begin{align*}
\lambda \left(h_n(B) \,\Delta \, h(B) \right)&=\lambda \left(\left(h_n^{-1}\right)^{-1}(h^{-1}(A)) 
\,\Delta \, \left(h_n^{-1}\right)^{-1} h_n^{-1}(A) \right) \\
&= \lambda \left(h^{-1}(A) \, \Delta \, h_n^{-1}(A) \right),
\end{align*}
so $\lim_{n \rightarrow \infty} \lambda \left(h_n(B) \,\Delta \, h(B) \right)=0$ if and only if
$\lim_{n \rightarrow \infty} \lambda \left(h_n^{-1}(A) \, \Delta \, h^{-1}(A)\right)=0$.
This shows that the third and the fourth assertion are equivalent, which, using 
Lemma \ref{lem:L1.in.prob} completes the proof.   
\end{proof}
Summing up, we can prove the following main result of this section:
\begin{thm}
$(\cdb^{mcd},D_1)$ and $(\mathcal{T}^*,\Vert \cdot \Vert_1)$ are homeomorphic to $(\ell_2,\Vert \cdot \Vert_2)$. 
\end{thm}
\begin{proof}
Define $\alpha: (\mathcal{T}^*,\Vert \cdot \Vert_1) \to (\cdb^{mcd},D_1)$ as the mapping assigning each 
$h \in \mathcal{T}^*$ its corresponding mutually completely dependent copula $C_h \in \cdb^{mcd}$. 
According to \cite{p06}, $\alpha$ is a surjective isometry, hence $\alpha$ is a homeomorphism.
Using the previous lemmas and the main result in \cite{Nhu}, it therefore follows that 
$(\cdb^{mcd},D_1)$ and $(\mathcal{T}^*,\Vert \cdot \Vert_1)$ are homeomorphic to 
$(\ell_2,\Vert \cdot \Vert_2)$. 
\end{proof}
Again using the fact that no, non-degenerate open ball in $(\ell,\Vert \cdot \Vert_2)$ 
has compact closure, the following corollary is immediate:
\begin{cor}
$(\cdb^{mcd},D_1)$ is nowhere locally compact.
\end{cor}
We conclude this section by showing that $(\cdb^{cd},D_1)$ is topologically very small in 
$(\cdb,D_1)$ - notice that, contrary to $(\cdb^{cd},D_1)$, the set $(\cdb^{mcd},D_1)$ is not closed 
and therefore can not be a $Z$-set in $(\cdb,D_1)$. 
\begin{thm}
$\cdb^{cd}$ is a $Z$-set in $(\cdb,D_1)$. 
\end{thm} 
\begin{proof}
First of all, contrary to $\cdb^{mcd}$, according to \cite{p06} $(\cdb^{cd},D_1)$ is closed 
in $(\cdb,D_1)$.\\
Let $\varepsilon: (\cdb,D_1) \to (0,\infty)$ denote an arbitrary but fixed continuous function.  Defining the mapping $\varepsilon': (\cdb,D_1) \to (0,1]$ by $\varepsilon'(C)=\min(\varepsilon(C),1)$, 
obviously $\varepsilon'$ is continuous too. 
We reuse the gluing $\mathcal{G}_a$ according to eq. \eqref{eq:glue}, and define the mapping 
$f: \cdb \to \cdb$ by
$$
f(C):=\mathcal{G}_{1-\frac{\varepsilon'(C)}{4}}(C,M).
$$
Proceeding as in the proof of Theorem \ref{thm:main.D1.full} it is, firstly, straightforward to 
verify that the mapping $f$ is continuous w.r.t. $D_1$ and that, secondly, 
$D_1(f(C),C)<\varepsilon'(C)\leq \varepsilon(C)$ holds for every $C \in \cdb$. 
Since obviously $f$ maps $\cdb$ into $\cdb \setminus \cdb^{cd}$, this completes the proof.
\end{proof}

\subsection{Bivariate Extreme-Value Copulas}
We quickly recall the definition and properties of Extreme Value copulas. 
A copula $C \in \mathcal{C}$ is called Extreme Value copula (EVC) if there exists some copula $B \in \mathcal{C}$ such that
$$
C(x,y)  = \lim_{n \rightarrow \infty} B^n(x^\frac{1}{n},y^\frac{1}{n})
$$
for all $x,y \in [0,1]$. Throughout this section the space of all bivariate EVCs will be denoted by 
$\cdb^{ev}$. It is well-known (see \cite{DiTru,TSFS} and the original references 
therein) that $(\cdb^{ev},d_\infty)$ is a compact metric space and that every EVC  
corresponds to a unique Pickands dependence function, i.e., a convex function $A: [0,1] \to [0,1]$
fulfilling $\max\{1-x,x\} \leq A(x) \leq 1$ for every $x \in [0,1]$. 
Letting $\mathcal{A}$ denote the family of all Pickands dependence functions, the 
one-to-one correspondence between $\cdb^{ev}$ and $\mathcal{A}$ is established via
    \begin{equation}\label{eq:map_eq_pick_copula}
    C(x,y) = (xy)^{A\left(\frac{\log(x)}{\log(xy)}\right)}, \quad x,y \in (0,1).
    \end{equation}
Writing $C_A$ for the (unique) EVC induced by $A \in \mathcal{A}$, 
the afore-mentioned one-to-one correspondence is easily seen to be continuous, i.e., 
the mapping $\kappa: \mathcal{A} \to \cdb^{ev}$, defined by $\kappa(A)=C_A$, is a 
homeomorhism between $(\mathcal{A},\Vert \cdot \Vert_\infty)$ and $(\cdb^{ev},d_\infty)$. 
More importantly, $\kappa$ is even a homeomorphism with respect to $D_1$ - the following result 
(in which $A^+$ denotes the right-hand derivative of $A \in \mathcal{A}$) established in 
\cite{KFT} holds: 
\begin{thm}[\cite{KFT}]\label{thm:wcc.evc}
Let $C, C_1, C_2, \ldots$ be Extreme Value copulas with Pickands dependence functions $A, A_1, A_2, \ldots$, respectively. Then the following assertions are equivalent:
\begin{enumerate}
    \item[(a)] $\lim_{n\to\infty} C_n(x,y) = C(x,y)$ for all $x, y \in [0,1]$,
    \item[(b)] $\lim_{n\to\infty} A_n(x) = A(x)$ for all $x \in [0,1]$,
    \item[(c)] $\lim_{n\to\infty} D^+A_n(x) = D^+A(x)$ for all continuity points $x$ of $A^+$,
    \item[(d)] $\lim_{n\to\infty} D_1(C_n, C) = 0$,
    \item[(e)] For $\lambda$-almost every $x \in (0,1)$ we have that 
    $K_{C_n}(x,\cdot) \xrightarrow{\ weakly\ } K_C(x,\cdot)$ for $n \to \infty$.
\end{enumerate}
\end{thm}
Following \cite{Pick}, every Pickands dependence function $A$ can be further be 
identified with a so-called spectral measure $\nu$ on the unit simplex. 
Projecting $\nu$ onto $[0,1]$, according to \cite{DiTru}, $A$ corresponds to a unique
probability measure $\vartheta \in \mathcal{P}([0,1])$ fulfilling 
$$
\int_{[0,1]} x d\vartheta(x)= \tfrac{1}{2}.
$$
We will let $\mathcal{P}_\mathcal{A}$ denote the family of all these measures and refer to them 
as Pickands dependence measures.  
The one-to-one correspondence between $\mathcal{A}$ and $\mathcal{P}_\mathcal{A}$ is established via 
the identity 
\begin{equation}\label{eq:id_pick_meas}
A(t) = 1-t + 2\int_{[0,t]}\vartheta([0,z]) \mathrm{d}\lambda(z),\quad t \in [0,1].
\end{equation}
Considering the Wasserstein metric $m_W$ on $\mathcal{P}_\mathcal{A}$, 
it is straightforward to verify that the corresponding mapping 
$\Upsilon: (\mathcal{P}_\mathcal{A},m_W) \to (\mathcal{A},\Vert \cdot \Vert_\infty)$ is a 
homeomorphism that preserves convex combinations.

The following lemma will be key for proving the main result of this section. 
\begin{lem}\label{lem:pick}
$\mathcal{A}$ is a compact, convex subset of $(C([0,1]),\Vert \cdot \Vert_\infty)$ containing
 infinitely many linearly independent elements. 
Moreover, $(\mathcal{A},\Vert \cdot \Vert_\infty)$ is homeomorphic to the Hilbert cube 
$(\mathcal{H},\rho)$.
\end{lem}
\begin{proof}
Showing that $\mathcal{A}$ contains infinitely many linearly independent elements 
can be done as follows:  
For every $a \in (0,\frac{1}{2})$ define the measure 
$\vartheta_a:=\frac{1}{2}(\epsilon_a + \epsilon_{1-a}) \in \mathcal{P}_\mathcal{A}$.
Obviously this family fulfills that for $a\neq b$ the supports of $\vartheta_a$ and $\vartheta_b$ 
are disjoint. 
As a direct consequence, the family $\{\vartheta_a: a \in (0,\frac{1}{2})\}$ is linearly 
independent in the 
vector space $\mathcal{S}([0,1])$ of all finite signed measures on $\mathcal{B}([0,1])$. 
In fact, if for $0<a_1<a_2 < \ldots < a_n < \frac{1}{2}$ there exist some 
constants $e_1,e_2,\ldots,e_n \in \mathbb{R}$ such that 
$$
\vartheta(G)=\sum_{i=1}^n e_i \vartheta_{a_i}(A)= 0
$$
for every $G \in \mathcal{B}([0,1])$ then, using disjointness of the supports
it follows immediately that $e_1=e_2=\ldots=e_n=0$. \\
Having this, applying Theorem \ref{thm:keller} yields that 
$(\mathcal{A},\Vert \cdot \Vert_\infty)$ is homeomorphic to $(\mathcal{H},\rho)$.  
\end{proof}

\begin{thm}
$(\cdb^{ev},D_1)$ is homeomorphic to the Hilbert cube $(\mathcal{H},\rho)$. In particular, 
$(\cdb^{ev},D_1)$ is compact, an AR, and has the DC and the FP property. 
Moreover, $\cdb^{ev}$ is a $Z$-set in $(\cdb,D_1)$
\end{thm}
\begin{proof}
Theorem \ref{thm:wcc.evc} implies that the afore-mentioned mapping 
$\kappa: (\mathcal{A},\Vert \cdot \Vert_\infty) \to (\cdb^{ev},D_1)$ 
is a homeomorphism, which directly yields that $(\cdb^{ev},D_1)$ is homeomorphic to the Hilbert 
cube $(\mathcal{H},\rho)$. Since AR, the DC and the FP property are direct consequences, it 
remains to prove that $\cdb^{ev}$ is a $Z$-set in $(\cdb,D_1)$. We will proceed in several steps
and let $\varepsilon: (\cdb,D_1) \to (0,\infty)$ denote an arbitrary but fixed continuous function. \\
\textbf{(S1)} Defining $\varepsilon': (\cdb,D_1) \to (0,1]$ by $\varepsilon'(C)=\min(\varepsilon(C),1)$, 
obviously $\varepsilon'$ is continuous too. 
Based on $\varepsilon'$ define the mapping: $\iota: \cdb \to \cdb$ by 
\begin{equation}\label{eq:iota}
\iota(C):=\left(1-\tfrac{\varepsilon'(C)}{2}\right)\,C + \tfrac{\varepsilon'(C)}{2}\,W.
\end{equation}
For showing that $\iota$ is continuous assume that $(C_n)_{n \in \mathbb{N}}$ is a sequence in 
$\cdb$ converging to $C \in \cdb$ w.r.t. $D_1$. Using the triangle inequality several times we obtain
\begin{align*}
D_1(\iota(C_n),\iota(C)) &\leq D_1(C_n,C) + \tfrac{1}{2}\vert \varepsilon'(C_n)-\varepsilon(C) \vert \\
& \quad + 
\tfrac{1}{2}\vert \varepsilon'(C_n)-\varepsilon(C)\vert \int_{[0,1]^2} K_{C_n}(x,[0,y]) d\lambda_2(x,y) \\
& \quad + \tfrac{\varepsilon'(C)}{2}\vert \int_{[0,1]^2} \vert K_{C_n}(x,[0,y]) - K_{C}(x,[0,y]) d\lambda_2(x,y) \\ 
& \leq \vert \varepsilon'(C_n)-\varepsilon(C)\vert\, + \, \tfrac{3}{2}D_1(C_n,C). 
\end{align*} 
Since the last two sums tend to $0$ for $n \to \infty$, continuity of $\iota$ follows. \\
\textbf{(S2)}  
We want to show that $\iota(C) \in \cdb \setminus \cdb^{ev}$ 
for every $C \in \cdb$. Setting 
$$
G:=\{(x,1-x): x \in [0,1]\}
$$ 
we obviously have $\mu_{\iota(C)}(G)>0$ for every $C \in \cdb$. No EVC, however, assigns
positive mass to $G$ - in fact, according to \cite{DiTru,TSFS} the Markov Kernel 
$K_{C_A}$ of every EVC $C_A$ with corresponding Pickands dependence measure $\vartheta$ 
fulfills the following equivalence for every $x,y \in (0,1)$:
\begin{align}
K_{C_A}(x,\{y\}) >0 \quad \Longleftrightarrow \quad t:=\tfrac{\ln(x)}{\ln(xy)} \text{ fulfills }
\vartheta(\{t\})>0.
\end{align} 
In particular, $K_{C_A}(x,\{1-x\}) >0 $ if and only if 
$\vartheta(\{\frac{\ln(x)}{\ln(x(1-x))}\})>0$. Considering that, firstly, 
$x \mapsto \frac{\ln(x)}{\ln(x(1-x))}$ is a strictly decreasing homeomorphism of $[0,1]$, and that, 
secondly, $\vartheta$ can at most have countably many point masses, it follows that 
there are at most countably many $x \in (0,1)$ with $K_{C_A}(x,\{1-x\}) >0 $. 
Using disintegration of $G$ therefore yields
\begin{align*}
\mu_{C_A}(G)&= \int_{[0,1]} K_{C_A}(x,G_x) d\lambda(x) = \int_{[0,1]} K_{C_A}(x,\{1-x\}) d\lambda(x)=0.
\end{align*}  
In other words, $\iota$ maps $\cdb$ into $\cdb \setminus \cdb^{ev}$. \\
\textbf{(S3)} As final step we need to show that $D_1(\iota(C),C)<\varepsilon'(C)$ holds for 
every $C \in \cdb$. Again using the triangle inequality and Fubini's theorem yields
\begin{align*}
D_1(\iota(C),C) &\leq \tfrac{\varepsilon'(C)}{2} \underbrace{\int_{[0,1]^2} K_{C}(x,[0,y]) d\lambda_2(x,y)}_{=\frac{1}{2}} + \tfrac{\varepsilon'(C)}{2} \underbrace{\int_{[0,1]^2} K_{W}(x,[0,y]) d\lambda_2(x,y)}_{=\frac{1}{2}} \\
&= \tfrac{2\varepsilon'(C)}{4}=\tfrac{\varepsilon'(C)}{2}< \varepsilon'(C).
\end{align*} 
Since $C \in \cdb$ was arbitrary this completes the proof. 
\end{proof}

\subsection{Spatially homogeneous copulas}
We quickly recall the definition and basic properties of the family of 
spatially homogeneous copulas. These copulas were motivated by \cite{Brown1966} in 
the context of Markov operators and studied in the (bivariate and multivariate) copula
setting in \cite{DuranteFernandezTrutschnig2020}. 
For every $a \in [0,1]$ we will let $r_a: [0,1] \to [0,1)$ denote the mapping (rotation), given by
$r_a(x)= (x+a)\text{mod1}$. Obviously we have $r_0=r_1$, moreover it is straightforward to 
verify that each $r_a$ is 
$\lambda$-preserving, i.e., that $\lambda^{r_a}=\lambda$ holds. 
Following \cite{DuranteFernandezTrutschnig2020} we will call a copula $C \in \cdb$ 
spatially homogeneous if and only if there exists some $\vartheta \in \mathcal{P}([0,1])$
such that 
\begin{align}
K(x,E) := \vartheta^{r_x}(E)= \vartheta\left(r_x^{-1}(E) \right), \quad x \in [0,1],\, 
E \in \mathcal{B}([0,1])
\end{align}
is a (version of the) Markov kernel of $C$. In order to indicate the dependence on $\vartheta$, in 
what follows we let $C_\vartheta \in \cdb$ denote the spatially homogeneous copula induced 
by $\vartheta$ and write $\cdb^{sh}$ for the family of all bivariate spatially homogeneous copulas. 
Figure \ref{fig:sh} depicts a sample of the spatially homogeneous copula $C_\beta$ with 
$\beta$ denoting the Beta-distribution with parameters $(2,3)$. 
It is straightforward to verify that considering $\vartheta=\lambda$ yields $C_\vartheta=\Pi$
and $\vartheta=\delta_{0}$ yields $C_\vartheta=M$. 
Obviously the family $\cdb^{sh}$ is convex - in order to show compactness in $(\cdb,D_1)$ 
we will use the subsequent lemma, the proof of which can be found in \cite{DFST2026}. 
\begin{figure}[htbp]
  \centering
  \includegraphics[width=0.7\textwidth]{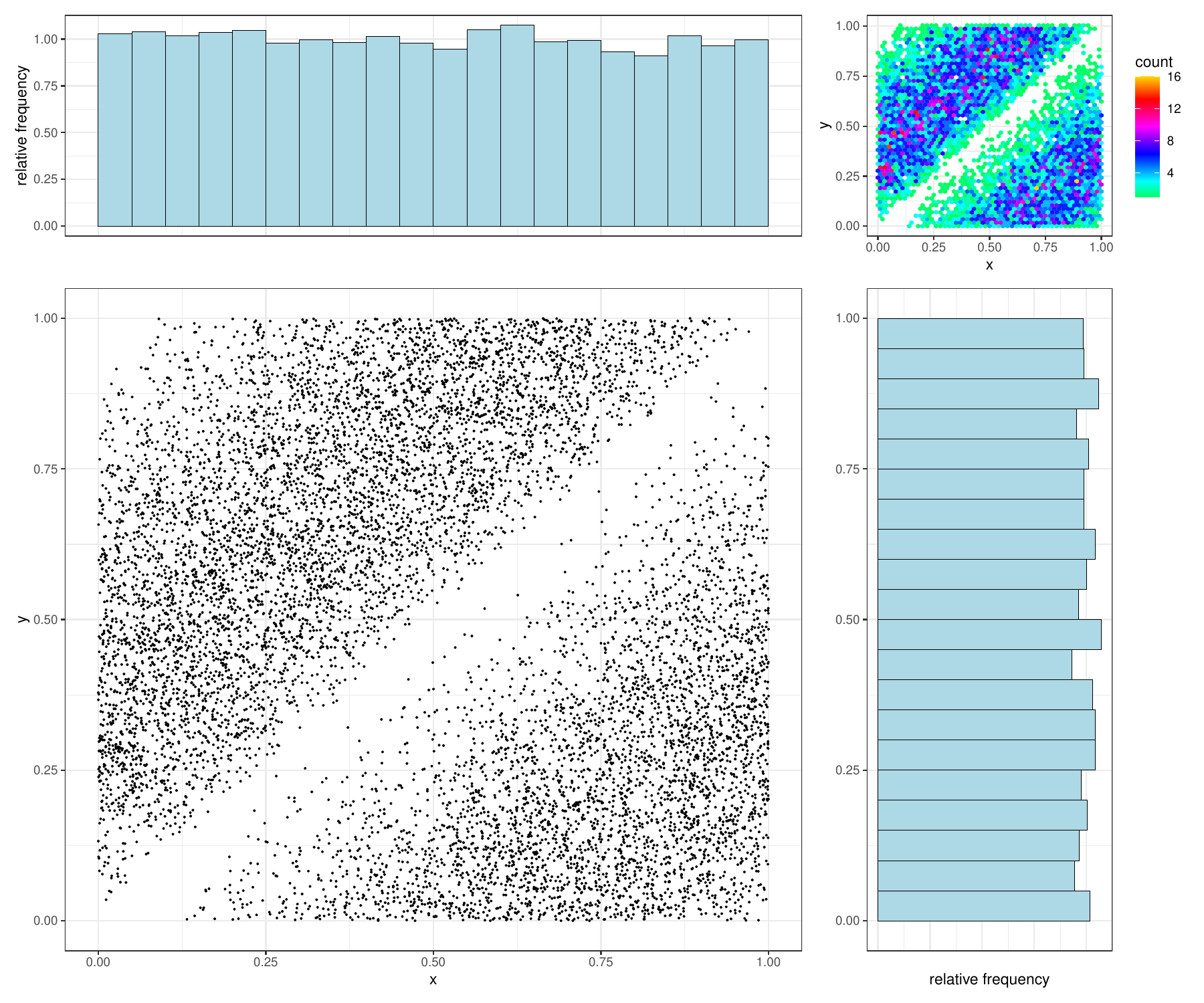}
  \caption{Sample of size $n=10.000$ of the spatially homogeneous copula 
  $C_\beta$ with $\beta$ denoting the Beta-distribution with parameters $(2,3)$ (upper
  left panel); corresponding marginal histograms (upper left and lower right panel)
  and heatmap of the sample (upper right panel).}
  \label{fig:sh}
\end{figure}
\begin{lem}\label{lem:weak.to.sh}
Suppose that $(\vartheta_n)_{n \in \mathbb{N}}$ is a sequence of probability measures in 
$\mathcal{P}([0,1])$ that converges weakly to $\vartheta \in \mathcal{P}([0,1])$. 
Then the corresponding induced spatially homogeneous copulas $C_{\vartheta_n}$ converge to 
$C_{\vartheta}$ with respect to $D_1$ for $n \to \infty$. 
\end{lem} 

\begin{thm}\label{lem:main.sh}
$\cdb^{sh}$ is a compact, convex subset of $(\cdb,D_1)$ containing infinitely many 
linearly independent elements. Moreover, $(\cdb^{sh},D_1)$ is homeomorphic to the 
Hilbert cube $(\mathcal{H},\rho)$, is an AR, has the DC and the FP property, and is a 
$Z$-set in $(\cdb,D_1)$.
\end{thm}

\begin{proof}
Convexity of $\cdb^{sh}$ is obvious - in fact, the mapping $\iota: \mathcal{P}([0,1]) \to \cdb^{sh}$, 
given by $\iota(\vartheta)=C_\vartheta$ preserves convex combinations. Furthermore, letting 
$m_W$ denote the Wasserstein metric (as metrization of weak convergence) on $\mathcal{P}([0,1])$, 
Lemma \ref{lem:weak.to.sh} yields continuity of $\iota$ w.r.t. the metrics $m_W$ and $D_1$. 
Considering that $(\mathcal{P}([0,1]),m_W)$ is a compact metric space (see, e.g., \cite{Bill,Kallenberg}) and that continuous images of compact metric spaces are compact too, it 
follows that $\cdb^{sh}$ is a compact, convex subset of $(\cdb,D_1)$. \\
Showing that $\cdb^{sh}$ contains infinitely many linearly independent elements (linearly independent in the Banach space $L_1([0,1]^2)$) is straightforward, since the family 
$\{C_{\delta_a}: a \in (0,1)\}$ is linearly independent. In fact, if for 
$0<a_1<a_2<\ldots< a_n <1$ there are constants $e_1,\ldots,e_n \in \mathbb{R}$ such that 
$\sum_{i=1}^n e_i K_{C_{\delta_{a_i}}}(u,[0,v])=0$ for $\lambda$-almost every $(u,v) \in [0,1]^2$, 
integrating over $[0,x] \times [0,y]$ yields that 
$$
\sum_{i=1}^n e_i C_{\delta_{a_i}}(x,y)=0
$$
holds for every $(x,y) \in [0,1]^2$. Translating to the corresponding doubly stochastic measures, 
it follows that the finite signed measure $\nu$, given by 
$$
\nu(G):= \sum_{i=1}^n e_i \, \mu_{C_{\delta_{a_i}}}(G),\quad G \in \mathcal{B}([0,1]^2)
$$  
is the zero measure, i.e., $\nu(G)=0$ holds for every $G \in \mathcal{B}([0,1]^2)$.  
Considering the special case of 
$G=\Gamma(r_{a_1}):=\{(x,r_{a_1}(x)): \, x \in [0,1]\} \in \mathcal{B}([0,1]^2)$, we have
\begin{align*}
0=\nu(\Gamma(r_{a_1}))= \sum_{i=1}^n e_i \, \mu_{C_{\delta_{a_i}}}(\Gamma(r_{a_1})) = 
e_1 \mu_{C_{\delta_{a_1}}}(\Gamma(r_{a_1})) = e_1. 
\end{align*}   
Proceeding analogously eventually yields $e_1=e_2=\ldots=e_n=0$, implying that the family 
$\{C_{\delta_a}: a \in (0,1)\}$ is linearly independent. As direct consequence we 
obtain that $\operatorname{ind}\cdb^{sh}=\infty$. \\
Summing up, we have shown that $\cdb^{sh}$ (or, more precisely, $\kappa(\cdb^{sh})$) is a 
convex compact subset of the Banach space $L_1([0,1]^2)$, which has small inductive dimension infinity.
Applying Theorem \ref{thm:keller} therefore yields that $(\cdb^{sh},D_1)$ is homeomorphic to the 
Hilbert cube $(\mathcal{H},\rho)$, which in turn implies that it is an AR, has the DC and the 
FP property.\\
It remains to show that $\cdb^{sh}$ is a $Z$-set in $(\cdb,D_1)$. Reusing $\iota$ as defined in 
eq. \eqref{eq:iota}, however, no additional effort is needed, since $\iota$ is easily verified
to map $\cdb$ into $\cdb \setminus \cdb^{sh}$ (no spatially homogeneous copula assign positive mass
to the set $G:=\{(x,1-x): x \in [0,1]\}$).    
\end{proof}

\noindent \textbf{Acknowledgement:}
The second author gratefully acknowledge the support of the project `IDA Lab Salzburg'
(20102/F2300464-KZP; 20204-WISS/225/348/3-2023).


\end{document}